\documentclass[MSNbibl,number,citesort,seceqn,dvips]{arxbj}
\usepackage{upgreek}
\usepackage{graphicx}

\aid{0}
\volume{19}
\issue{5B}
\pubyear{2013}
\firstpage{2750}
\lastpage{2767}
\doi{10.3150/12-BEJ473} 

\makeatletter
\renewcommand{\pi}{\uppi}
\newtheorem{thmm}{Theorem}[section]
\newtheorem{prop}[thmm]{Proposition}
\newremark{exa}[thmm]{Example}
\newremark{rem}[thmm]{Remark}
\newproclaim{defi}[thmm]{Definition}

\newcommand{\real}{\mathbb{R}}
\newcommand{\comp}{\mathbb{C}}
\newcommand{\nat}{\mathbb{N}}
\newcommand{\im}{\operatorname{Im}}
\newcommand{\re}{\operatorname{Re}}

\newcommand{\iu}{\mathcal{UI}}

\makeatother

\begin{document}
\begin{frontmatter}

\title{On a class of explicit Cauchy--Stieltjes transforms related to
monotone stable and free Poisson laws}
\runtitle{Cauchy--Stieltjes transforms related to monotone stable and
free Poisson laws}

\begin{aug}
\author[a]{\fnms{Octavio} \snm{Arizmendi}\corref{}\thanksref{a}\ead[label=e1]{arizmendi@math.uni-sb.de}}
\and
\author[b]{\fnms{Takahiro} \snm{Hasebe}\thanksref{b}\ead[label=e2]{thasebe@math.kyoto-u.ac.jp}}
\runauthor{O. Arizmendi and T. Hasebe} 
\address[a]{Universit\"{a}t des Saarlandes, FR 6.1-Mathematik, 66123
Saarbr\"{u}cken, Germany.\\ \printead{e1}}
\address[b]{Graduate School of Science, Kyoto University, Kyoto
606-8502, Japan.\\ \printead{e2}}
\end{aug}

\received{\smonth{8} \syear{2011}}
\revised{\smonth{8} \syear{2012}}

%
\begin{abstract}
We consider a class of probability measures $\mu_{s,r}^{\alpha}$
which have explicit Cauchy--Stieltjes transforms. This class includes a
symmetric beta distribution, a free Poisson law and some beta
distributions as special cases. Also, we identify $\mu_{s,2}^{\alpha
}$ as a free compound Poisson law with L\'evy measure a monotone $\alpha
$-stable law. This implies the free infinite divisibility of $\mu_{s,2}^{\alpha}$.
Moreover, when symmetric or positive, $\mu_{s,2}^{\alpha}$ has a representation as the free multiplication of a
free Poisson law and a monotone $\alpha$-stable law. We also
investigate\vspace*{1pt} the free infinite divisibility of $\mu_{s,r}^{\alpha}$
for $r\neq2$. Special cases include the beta distributions $B(1-\frac
{1}{r},1+\frac{1}{r})$ which
are freely infinitely divisible if and only if $1 \leq r \leq2$.
\end{abstract}

%

\end{frontmatter}

\section{Introduction}
In random matrix theory, a Marchenko--Pastur law describes the
asymptotic behavior of
the spectrum of the so-called Wishart matrices \cite{MP}.
In free probability, a Marchenko--Pastur (or free Poisson) law plays the
role that a Poisson distribution does in probability theory: it is the
limiting distribution of $((1-\frac{\lambda}{N})\delta_{0}+\frac
{\lambda}{N}\delta_{1})^{\boxplus N}$ when $N \to\infty$. For this
reason, it is called a free Poisson law in the context of free probability.
On the other hand, an arcsine law appears in probability theory as the
law of the proportion of the time during which a Wiener process is
nonnegative. In monotone probability, an arcsine law plays the role of
a Gaussian law \cite{Mur2}. In particular, an arcsine law is a monotone
stable law with stability index $\alpha=2$ \cite{Has2}.

Arizmendi \textit{et al.} \cite{ABNPA09} found an interplay between
Marchenko--Pastur and arcsine laws. They introduced a class $\mathit{FTA}$ of
freely infinitely divisible distributions whose L\'{e}vy measures are
mixtures of a symmetric arcsine law. The building block of this class
is a symmetric beta distribution
\[
b_{s}(\mathrm{d}x) = \frac{1}{\pi\sqrt{s}}|x|^{-1/2}\bigl(\sqrt
{s}-|x|\bigr)^{1/2}\,\mathrm{d}x, \qquad -\sqrt{s} \leq x \leq\sqrt{s}.
\]
The free L\'evy measure of $b_s$ coincides with an arcsine law.
Moreover, $b_s$ is equal to the free multiplicative convolution of an
arcsine law with a Marchenko--Pastur law, and hence, is freely
infinitely divisible. Moreover, its Cauchy--Stieltjes transform (or
Cauchy transform for short) can be calculated explicitly as
%
\begin{equation}
\label{eq0111} G_{b_s}(z)=-\sqrt{\frac{2}{s}}\sqrt{1-
\sqrt{1-sz^{-2}}},\qquad  s > 0.
\end{equation}
%

This paper studies a class of Cauchy--Stieltjes (or Cauchy for short)
transforms related to Marchenko--Pastur laws and monotone stable laws.
We deform the above Cauchy transform (\ref{eq0111}) to introduce a
family of probability measures which include the symmetric beta
distribution~$b_s$, Marchenko--Pastur and some other beta distributions
as special cases. More explicitly, for $0 < \alpha\leq2$, we define
%
\begin{equation}
\label{eq46} G_{s,r}^{\alpha}(z)=-r^{1/\alpha} \biggl(
\frac{1-(1-s(-
{1}/{z})^{\alpha})^{1/r}}{s} \biggr)^{1/\alpha},\qquad  r>0, s\in\mathbb {C}\setminus\{0\}.
\end{equation}
The branches of powers have to be defined carefully and the precise
definition is presented in Section \ref{sec2}. It can be shown that the
function (\ref{eq46}) defines the Cauchy transform of a probability
measure $\mu_{s,r}^{\alpha}$ for $1\leq r<\infty$ and $(\alpha,s)$
satisfying what we call an admissible condition. This condition is
related to stable distributions.

The reciprocal Cauchy transforms $F_{s,r}^{\alpha}=\frac
{1}{G_{s,r}^{\alpha}}$ satisfy
\[
F_{s,r}^{\alpha}\circ F_{us,u}^{\alpha}=F_{us,ur}^{\alpha}.
\]
We note that the same relation appears for probability measures
introduced by M\l otkowski \cite{Mlot}. This relation enables us to
calculate the inverse map explicitly:
%
\begin{equation}
\label{eq1} \bigl(F^{\alpha}_{s,r}\bigr)^{-1} =
F^\alpha_{s/r,1/r}.
\end{equation}
The inverse map of the reciprocal Cauchy transform, which is hard to
calculate in general, is crucial to investigate free infinite
divisibility. Therefore, the explicit form of $(F^\alpha_{s,r})^{-1}$
is quite useful and we can prove the free infinite divisibility of $\mu^\alpha_{s,r}$ for some parameters.

The probability measure $\mu_{s,2}^{\alpha}$ turns out to be a free
compound Poisson distribution with L\'evy measure a monotone $\alpha
$-stable law $a_{s/4}^{\alpha}$. From Proposition 4 of \cite{P-S}, if
symmetric or positive, $\mu_{s,2}^{\alpha}$ coincides with the free
multiplicative convolution of a Marchenko--Pastur law $m$ and the
monotone $\alpha$-stable distribution $a_{s/4}^{\alpha}$:
\[
\mu_{s,2}^{\alpha}=m\boxtimes a_{s/4}^{\alpha}.
\]
Moreover, $\mu^{\alpha}_{s,r}$ is freely infinitely divisible for other
parameters, not only for $r=2$. An interesting case of $\mu^\alpha_{s,r}$ is
$\mu^{1}_{-1,r}$ which is a beta distribution with the
density $\frac{r\sin(\pi/r)}{\pi} x^{-1/r}(1-x)^{1/r}$ on $(0,1)$. We
prove that this is freely infinitely divisible if and only if $1 \leq r
\leq2$.
We also mention that, while an arcsine law is not freely infinitely
divisible, some monotone stable laws are. This fact was implicitly
proved by Biane in a different context; see Corollary 4.5 of \cite{Biane}.

\section{Preliminary results}\label{sec22}
\subsection{The Voiculescu transform and the $R$-transform}
In this paper, $\comp_+$ and $\comp_-$, respectively, denote the upper
half-plane and the lower half-plane of~$\comp$.

An additive free convolution $\mu\boxplus\nu$ of compactly supported
probability measures $\mu$ and $\nu$ on $\real$ is the probability
distribution of $X+Y$, where $X$ and $Y$ are self-adjoint free
independent random variables with distributions $\mu$ and $\nu$,
respectively, \cite{V2}. This convolution was extended to all Borel
probability measures in \cite{Be-Vo}. A probability measure $\mu$ on
$\real$ is said to be $\boxplus$-infinitely divisible if for any $n \in
\nat$, there is $\mu_n$ such that $\mu= \mu_n^{\boxplus n}$.

For a probability measure $\mu$ on $\real$, let us denote by $G_\mu$
the Cauchy transform and by $F_\mu$ its reciprocal: $G_\mu(z)=\int_{\real}\frac{\mu(\mathrm{d}x)}{z-x}$ and $F_\mu(z)=\frac{1}{G_\mu(z)}$.
Bercovici and Voiculescu \cite{Be-Vo} proved the existence of $\eta,\eta' >0$ and $M, M' >0$ such that $F_\mu$ is univalent in $\Gamma_{\eta
,M}:= \{z \in\comp_+\dvt  \im z >M,  |\im z| > \eta|\re z| \}$ and $\Gamma_{\eta',M'} \subset F_\mu(\Gamma_{\eta,M})$. The Voiculescu transform
$\phi_\mu$ is defined in $\Gamma_{\eta',M'}$ to be $F_\mu^{-1}(z)-z$.
The free convolution $\mu\boxplus\nu$ is characterized by
\[
\phi_{\mu\boxplus\nu}(z) = \phi_\mu(z) + \phi_\nu(z)
\]
in $\Gamma_{\eta'',M''}$ for some $\eta'',M''>0$. $R_\mu(z):=z\phi_\mu
(\frac{1}{z})$ is called an $R$-transform. A probability measure $\mu$
is $\boxplus$-infinitely divisible if and only if $\phi_\mu$ is the
restriction of an analytic map from $\comp_+$ into $\comp_- \cup\real
$ \cite{Be-Vo}. This is also equivalent to the L\'{e}vy--Khintchine type
representation suggested in~\cite{BNT06}
%
\begin{equation}
\label{eq00} R_\mu(z) = cz +az^2 + \int
_{\real} \biggl( \frac{1}{1-xz}-1-xz\textbf
{1}_{\{|x|\leq1\}}(x) \biggr) \nu(\mathrm{d}x),
\end{equation}
for some $c \in\real, a\geq0$ and a nonnegative measure $\nu$
satisfying $\nu(\{0\})=0$ and $\int_{\real} \min\{1,x^2 \}\times \nu(\mathrm{d}x) <
\infty$. We call $\nu$ the L\'{e}vy measure of $\mu$.

The following is useful to calculate the L\'{e}vy measure. For a
$\boxplus$-infinitely divisible measure $\mu$, its Voiculescu transform
can be written as
\[
\phi_\mu(z) = \gamma+ \int_{\mathbb{R}} \biggl(
\frac{1}{z-x}-\frac
{x}{1+x^2} \biggr) \bigl(1+x^2\bigr)
\tau(\mathrm{d}x)
\]
for some $\gamma\in\mathbb{R}$ and a nonnegative finite measure $\tau
$ \cite{Be-Vo}.
The measure $\tau$ can be calculated, by using the Stieltjes inversion
formula \cite{Akh,T}, as
%
\begin{equation}
\label{eqtau} \int_u^v \bigl(1+x^2
\bigr)\tau(\mathrm{d}x) = -\frac{1}{\pi} \lim_{y \searrow0}\int_{u}^v
\operatorname{Im} \phi_\mu(x+\mathrm{i}y)\,\mathrm{d}x
\end{equation}
for all continuity points $u ,v$ of $\tau$. Considering the relation
$R_\mu(z)=z\phi_\mu(\frac{1}{z})$ and (\ref{eq00}), we obtain $\frac
{1+x^2}{x^2}\tau|_{\mathbb{R}\setminus\{0\}}=\nu|_{\mathbb{R}\setminus
\{0\}}$ and $\tau(\{0\})=a$, where $a$ is the real number of (\ref
{eq00}). In particular, if the functions $f_\mu^y(x):=-\frac{1}{\pi
}\operatorname{Im} \phi_\mu(x+\mathrm{i}y)$ converges uniformly to a continuous
function $f_\mu(x)$ ($y \searrow0$) on an interval $[u,v]$, then $\tau
$ is absolutely continuous in $[u,v]$ with density $f_\mu(x)$. Hence,
$\nu$ is also absolutely continuous in $[u,v]$ with density
%
\begin{equation}
\label{eqle} \frac{1+x^2}{x^2}f_\mu(x).
\end{equation}
Regarding atoms, the following formula holds: $\tau(\{x \})= \frac
{1}{1+x^2}\lim_{y \searrow0} \mathrm{i}y \phi_\mu(x+\mathrm{i}y)$ for any \mbox{$x \in\real$.}

\subsection{The $S$-transform}
Multiplicative free convolution $\boxtimes$ for probability measures on
$[0,\infty)$ was investigated in \cite{Voi87,Be-Vo}. This convolution
corresponds to the probability distribution of $X^{1/2}Y X^{1/2}$, or
equivalently $Y^{1/2}XY^{1/2}$, where $X$ and $Y$ are positive free
independent random variables. This convolution is characterized by
$S$-transforms defined as follows. For a probability measure $\mu$ on
$\real$, we let $\psi_\mu(z):=\int_{\real}\frac{zx}{1-zx}\mu(\mathrm{d}x)$. $\psi_\mu$ coincides with a moment generating function if $\mu$ has finite
moments of all orders.
In \cite{Be-Vo}, $\psi_\mu$ was proved to be univalent in the left
half-plane $\mathrm{i}\comp_+$ for a probability measure $\mu$ on $[0,\infty)$
with $\mu(\{0\}) <1$. Moreover, $\psi_\mu(\mathrm{i}\comp_+)$ contains the
interval $(1-\mu(\{0\}),0)$. Then a map $\chi_\mu\dvt  \psi_\mu(\mathrm{i}\comp_+)
\to \mathrm{i}\comp_+$ is defined by the inverse of $\psi_\mu$. The
$S$-transform is defined as
%
\begin{equation}
\label{eq4} S_\mu(z):= \frac{1+z}{z}\chi_\mu(z),\qquad  z
\in \psi_\mu(\mathrm{i}\comp_+).
\end{equation}
Using the $S$-transform, $\mu\boxtimes\nu$ is characterized as
%
\begin{equation}
\label{eq5} S_{\mu\boxtimes\nu} (z) = S_\mu(z)S_\nu(z)
\end{equation}
in a common domain including an interval of the form $(-\varepsilon,0)$.

More generally, a multiplicative convolution $\mu\boxtimes\nu$ can be
defined if $\mu$ or $\nu$ is supported on $[0,\infty)$. While (\ref
{eq5}) is expected to hold also in this case, it is not known whether
an $S$-transform can be defined for every probability measure. It was
shown in \cite{Voi87} to hold for measures with bounded support and
nonvanishing mean, while the bounded case when $\mu$ has vanishing
mean was solved in \cite{RS}. For the unbounded case, as a partial
solution, Arizmendi and P\'erez-Abreu \cite{A-P} defined an $S$-transform of a
symmetric probability measure as follows. For a symmetric distribution
$\mu\neq\delta_0$, there is a unique probability distribution $\mu^2
\neq\delta_0$ on $[0,\infty)$ such that $\psi_\mu(z)=\psi_{\mu
^2}(z^2)$ for $z \in\comp_+$. Using a property of $\psi_{\mu^2}$, we
can conclude that $\psi_\mu$ is univalent in $\mathbb{H}:=\{z\in\comp_+\dvt \im z > |\re z| \}$.
 Moreover, $\psi_\mu(\mathbb{H})$ contains the
interval $(1-\mu(\{0\}),0)$. Therefore, we can define
$\chi_\mu= \psi_\mu^{-1}\dvt  \psi_\mu(\mathbb{H}) \to\mathbb{H}$ and
$S_\mu(z):= \frac{1+z}{z}\chi_\mu(z)$. Then (\ref{eq5}) still holds if
$\mu$ or $\nu$ is symmetric and the other is supported on $[0,\infty)$.

Finally, we recall the analogues of compound Poisson distributions,
which will be important in this paper.
%
\begin{defi}
A probability measure $\mu$ is said to be free compound Poisson if
$R_\mu(z) = \lambda\psi_\nu(z)$ for a probability measure $\nu$ with
$\nu(\{0\})=0$ and a $\lambda\geq0$. In this case,
$\lambda\nu$ coincides with the L\'evy measure of $\mu$.
\end{defi}
The Marchenko--Pastur law $m$ with mean one belongs to the class of free
compound Poisson measures; the pair $(\lambda, \nu)$ is given by $(1,
\delta_1)$.
$m$ is also characterized by $S_m(z)=\frac{1}{z+1}$ in terms of the
$S$-transform.

\section{\texorpdfstring{Probability measures $\mu_{s,r}^\alpha$}{Probability measures mu s,r alpha}}\label{sec2}
Let $r>0, 2 \geq\alpha>0$ and $s \in\comp\setminus\{0\}$. For any
$\eta>0$, we will find an $M>0$ such that the function
%
\begin{equation}
\label{eq8} G_{s,r}^{\alpha}(z)=-r^{1/\alpha} \biggl(
\frac{1-(1-s(-
{1}/{z})^{\alpha})^{1/r}}{s} \biggr)^{1/\alpha}
\end{equation}
is defined as an analytic map in $\Gamma_{\eta,M}$.
To make the definition precise, we take branches of powers $z^{1/\alpha
}$, $z^{1/r}$ and $z^{\alpha}$ as follows:
\begin{longlist}[(1)]
\item[(1)] $z^{1/\alpha}$ and $z^{\alpha}$ are, respectively, defined as
$\mathrm{e}^{{1}/{\alpha}\log_{(1)}z}$ and $\mathrm{e}^{\alpha\log_{(1)}z}$ in
$\comp\setminus[0,\infty)$, where $\log_{(1)}$ denotes a logarithm
satisfying $\im(\log_{(1)}z) \in(0,2\pi)$;
\item[(2)] $z^{1/r}$ is defined to be $\mathrm{e}^{{1}/{r}\log_{(2)}z}$ in
$\comp\setminus(-\infty,0]$, where $\log_{(2)}$ is a logarithm
satisfying $\im(\log_{(2)}z) \in(-\pi,\pi)$.
\end{longlist}
We show that these branches enable us to define $G_{s,r}^{\alpha}$ as
an analytic function in $\Gamma_{\eta,M}$ for an $M>0$ depending on
$\eta>0, s \in\mathbb{C}\setminus\{0\}, r >0$.
Under the definition (2), the function $(1+w)^{1/r}$ is equal to the
generalized binomial expansion $\sum_{n=0}^\infty{}_{1/r}C_n w^n$ for
$|w| < 1$, where ${}_{1/r}C_n$ is the generalized binomial coefficient
$\frac{1/r(1/r-1)\cdots(1/r-n+1)}{n!}$. Therefore, for $z \in\comp_+$
with large $|z|$, the function $ \frac{1-(1-s(-{1}/{z})^{\alpha
})^{1/r}}{s}$ can be written as
%
\begin{equation}
\label{eq403} \frac{1-(1-s(-{1}/{z})^{\alpha})^{1/r}}{s}= \biggl(-\frac{1}{z}
\biggr)^\alpha\sum_{n=1}^\infty{}_{1/r}C_n
(-s)^{n-1} \biggl(-\frac
{1}{z} \biggr)^{(n-1)\alpha},
\end{equation}
where $(-\frac{1}{z})^{n\alpha}$ is defined by $ ((-\frac
{1}{z})^\alpha )^n$. For any $\eta>0, s \in\mathbb{C}\setminus\{
0\}, r>0$, there is an $M>0$, independent of $\alpha\in(0,2]$, such
that the image of the map $\Gamma_{\eta,M} \ni z \mapsto (-\frac
{1}{z} )^\alpha\sum_{n=1}^\infty{}_{1/r}C_n (-s)^{n-1}\times
(-\frac{1}{z} )^{(n-1)\alpha}$ is contained in the sector $\{z \in
\comp\setminus\{0\}\dvt  \arg z \in(0, \alpha\pi)\}$. Therefore, we can
take the power of (\ref{eq403}) by $1/\alpha$ and $G^\alpha_{s,r}$ is
well-defined as an analytic map in $\Gamma_{\eta,M}$.

We note that $G_{s,r}^{\alpha}(z)$ can be expanded in a series
regarding $ (-\frac{1}{z} )^\alpha$:
%
\begin{eqnarray}
\label{eq7} %
G_{s,r}^{\alpha}(z) &=& -r^{1/\alpha}
\Biggl( \biggl(-\frac{1}{z} \biggr)^\alpha\sum
_{n=1}^\infty{}_{1/r}C_n
(-s)^{n-1} \biggl(-\frac
{1}{z} \biggr)^{(n-1)\alpha}
\Biggr)^{1/\alpha}
\nonumber\\
&=& \frac{1}{z} \Biggl( 1+ r\sum_{n=1}^\infty{}_{1/r}C_{n+1}
(-s)^{n} \biggl(-\frac{1}{z} \biggr)^{n\alpha}
\Biggr)^{1/\alpha}
\\
&=&\frac{1}{z}\sum_{n=0}^\infty
c_n(\alpha, s, r) \biggl(-\frac
{1}{z} \biggr)^{n\alpha},\qquad
z\in\Gamma_{\eta,M} \nonumber
\end{eqnarray}
for some complex coefficients $c_n(\alpha, s ,r)$ with $c_0=1$. In the
second line, we used the formula $ ( (-\frac{1}{z}
)^{\alpha}(1+\mathrm{o}(1/z)) )^{1/\alpha} = -\frac
{1}{z}(1+\mathrm{o}(1/z))^{1/\alpha}$. This formula is valid in $\Gamma_{\eta
,M}$ if the function $ ( 1+ \mathrm{o}(1/z)  )^{1/\alpha}$ is
understood to be the generalized binomial expansion.

Let us define $F^\alpha_{s,r}(z):=\frac{1}{G^\alpha_{s,r}(z)}$ for $z
\in\Gamma_{\eta,M}$, where $M>0$ is large enough depending on $(\eta
,s,r)$. Then we have the following.
%
\begin{thmm}\label{thm001} Let $r,u > 0, 2 \geq\alpha>0, \eta>0$ and
$s \in\comp\setminus\{0\}$. Then\vspace*{-1pt}
\[
F_{s,r}^{\alpha}\circ F_{us,u}^{\alpha}=
F_{us,ur}^{\alpha}\vspace*{-1pt}
\]
holds in $\Gamma_{\eta, M}$ for some $M>0$.
\end{thmm}
\begin{pf}
We note that $(-G^\alpha_{us,u}(z))^\alpha$ is equal to $\frac
{1-(1-us(-{1}/{z})^\alpha)^{1/u}}{s}$ in $\Gamma_{\eta,M}$ with
large $M>0$. Also, we note that
$ ((1+w)^{1/r} )^{1/u} = (1+w)^{1/(ru)}$ for small $|w|$. Then\vspace*{-1pt}
\[
-r^{1/\alpha} \biggl( \frac{1-(1-s(-G^\alpha_{us,u}(z))^\alpha
)^{1/r}}{s} \biggr)^{1/\alpha} =
-(ur)^{1/\alpha} \biggl( \frac
{1-(1-us(-{1}/{z})^\alpha)^{1/(ru)}}{us} \biggr)^{1/\alpha}\vspace*{-1pt}
\]
for $z \in\Gamma_{\eta,M}$.
\end{pf}

Under further conditions on $(r,\alpha,s)$, the function $G^\alpha_{s,r}$ is well-defined in $\comp_+$ with values in $\comp_-$, and
therefore defines a probability measure.\vspace*{-1pt}
%
\begin{thmm}\label{thm1} Suppose $1 \leq r < \infty$, $0 < \alpha\leq
2$ and $s \in\comp\setminus\{0\}$. Assume that either of the
following conditions is satisfied:\vspace*{-1pt}
\begin{longlist}[(1)]
\item[(1)] $0 < \alpha\leq1$ and $(1-\alpha)\pi\leq\arg s \leq\pi$;

\item[(2)] $1 < \alpha\leq2$ and $0 \leq\arg s \leq(2-\alpha)\pi$.\vspace*{-1pt}
\end{longlist}

\begin{figure}

\includegraphics{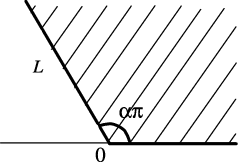}

\caption{The image of $\mathbb{C}_+$ under the map $z\mapsto(-\frac
{1}{z})^\alpha$.}\label{dia5}\vspace*{-3pt}
\end{figure}

Then $G_{s,r}^{\alpha}$ is the Cauchy transform of a probability
measure, which we denote by $\mu_{s,r}^\alpha$. Moreover,
$G_{s,r}^\alpha$ is univalent in $\comp_+$.
If $(\alpha, s)$ satisfies (1) or (2), it is said to be admissible.
\end{thmm}
\begin{pf}
Let $r \geq1$. We can immediately check that $z G^{\alpha}_{s,r}(z)
\to1$ as $z \to\infty, z\in\comp_+$, nontangentially. Therefore,
what needs to be proved is that $G^{\alpha}_{s,r}$ analytically maps
the upper half-plane to the lower half-plane.

We first focus on the case $0 < \alpha\leq1$ and $\theta:= \arg s
\in[\pi(1-\alpha), \pi]$. Then the image of the map
$\frac{1-(1-s(-{1}/{z})^{\alpha})^{1/r}}{s}$ in $\comp_+$ can be
described as in Figure~\ref{dia8} after some steps described in Figures~\ref{dia5}--\ref{dia7}.
We can see that the image of the map $\frac
{1-(1-s(-{1}/{z})^{\alpha})^{1/r}}{s}$ is contained in the sector
$\{z \in\comp\dvt  0< \arg z < \alpha\pi\}$. This implies the desired conclusion.

\begin{figure}

\includegraphics{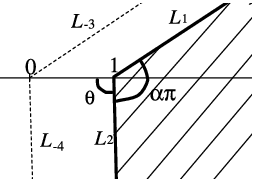}

\caption{The image of $\mathbb{C}_+$ under the map $z\mapsto1-s(-\frac
{1}{z})^\alpha$. $L_1$ and $L_2$ are half lines contained in the upper
half-plane and the lower half-plane, respectively. $L_{-3}$ and
$L_{-4}$ are preimages of $L_3$ and $L_4$ of Figure~\protect\ref{dia7} for the
map $z \mapsto z^{1/r}$, respectively.}\label{dia6}\vspace*{-3pt}
\end{figure}

\begin{figure}

\includegraphics{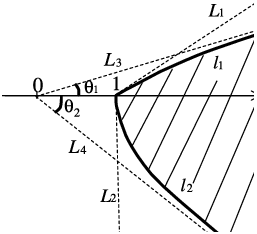}

\caption{The image of $\mathbb{C}_+$ under the map $z\mapsto(1-s(-\frac
{1}{z})^\alpha)^{1/r}$. $\theta_1$ and $\theta_2$ are defined by $\theta_1= \frac{\theta-(1-\alpha)\pi}{r}$ and $\theta_2 = \frac{\pi- \theta
}{r}$. $L_1$ and $L_2$ are the same half lines as in Figure~\protect\ref{dia6}.
$L_3$ and $L_4$ are starting at $0$. $l_1$ is tangent to $L_1$ at $1$
since $z^{1/r}$ is a conformal mapping. Moreover, it approaches $L_3$
asymptotically. $l_2$ is tangent to $L_2$ at $1$ from the same reason
and approaches $L_4$ asymptotically. }\label{dia7}
\end{figure}

\begin{figure}

\includegraphics{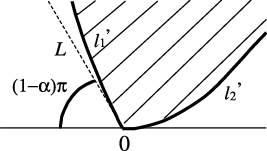}

\caption{The image of $\mathbb{C}_+$ under the map $z\mapsto\frac
{1-(1-s(-{1}/{z})^\alpha)^{1/r}}{s}$ which can be obtained from the
rotation and the translation of Figure~\protect\ref{dia7}. $l_1$' is tangent to
$L$ and $l_2$' is tangent to the $x$ axis at $0$.
}\label{dia8}\vspace*{-3pt}
\end{figure}

In the case $1 < \alpha\leq2$, we draw similar pictures; see Figure~\ref{dia9}--\ref{dia12}.
In Figure~\ref{dia12}, the image of $\frac
{1-(1-s(-{1}/{z})^{\alpha})^{1/r}}{s}$ is contained in the sector
$\{z \in\comp\dvt  0< \arg z < \alpha\pi\}$. Therefore, the image of the
map $ (\frac{1-(1-s(-{1}/{z})^{\alpha})^{1/r}}{s} )^
{{1}/{\alpha}}$ is contained in $\comp_+$.

In each step described in the figures, a new univalent map is added, so
that after all the steps, the map $G_{s,r}^\alpha$ is also univalent in
$\mathbb{C}_+$.
\end{pf}
%

\begin{figure}

\includegraphics{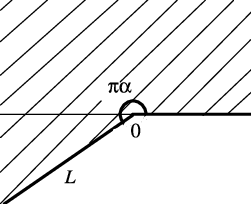}

\caption{The image of $\mathbb{C}_+$ under the map $z\mapsto(-\frac
{1}{z})^\alpha$.}\label{dia9}\vspace*{-3pt}
\end{figure}

\begin{figure}

\includegraphics{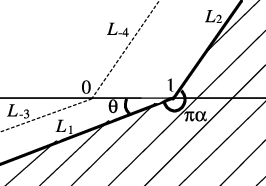}

\caption{The image of $\mathbb{C}_+$ under the map $z\mapsto1-s(-\frac
{1}{z})^\alpha$. $L_1$ and $L_2$ are half lines starting at $1$.
$L_{-3}$ and $L_{-4}$ are preimages of $L_3$ and $L_4$ of Figure~\protect\ref
{dia11} for the map $z \mapsto z^{1/r}$, respectively.}\label{dia10}
\end{figure}

\begin{figure}

\includegraphics{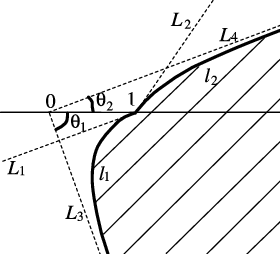}

\caption{The image of $\mathbb{C}_+$ under the map $z\mapsto(1-s(-\frac
{1}{z})^\alpha)^{1/r}$. $\theta_1$ and $\theta_2$ are defined by $\theta_1= \frac{\pi-\theta}{r}$ and $\theta_2 = \frac{\pi(\alpha-1) + \theta
}{r}$. $L_1$ and $L_2$ are the same half lines as in Figure~\protect\ref{dia10}.
$L_3$ and $L_4$ are starting at $0$. $l_1$ is tangent to $L_1$ at $1$
and approaches $L_3$ asymptotically. $l_2$ is tangent to $L_2$ at $1$
and approaches $L_4$ asymptotically. }\label{dia11}\vspace*{-3pt}
\end{figure}

\begin{figure}

\includegraphics{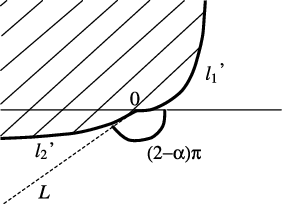}

\caption{The image of $\mathbb{C}_+$ under the map $z\mapsto\frac
{1-(1-s(-{1}/{z})^\alpha)^{1/r}}{s}$ which can be obtained from the
rotation and the translation of Figure~\protect\ref{dia11}. $l_1$' is tangent to
the $x$ axis and $l_2$' is tangent to $L$ at $0$.}\label{dia12}
\end{figure}

\begin{rem}\label{rem2}
(i) The admissible condition is related to monotone stable
distributions as mentioned in the next section.
\begin{longlist}[(iii)]
\item[(ii)] We have $\mu_{s,1}^\alpha=\delta_0$ for any admissible $(\alpha
,s)$. Therefore, the right inverse of $F^{\alpha}_{s,r}$ can be
calculated as $(F^{\alpha}_{s,r})^{-1} = F^\alpha_{s/r,1/r}$ from
Theorem \ref{thm001}.

\item[(iii)] From the relation $(F^\alpha_{s,r})^{-1} = F^\alpha_{s/r,1/r}$,
we can conclude that $G^\alpha_{s,r}$ does not define a probability
measure for $0 < r < 1$ and admissible
$(\alpha,s)$. The reason is as follows. If $\mu$ is a probability
measure and not a point mass, then $\operatorname{Im} F_\mu(z) > \operatorname{Im} z$
for any $z \in\comp_+$; see Corollary 5.3 of \cite{Be-Vo}. Hence,
$\operatorname{Im}  F_\mu^{-1}(z) < \operatorname{Im}  z$ if $z = F_\mu(w)$ and $F_\mu
$ is univalent around $w$. Therefore, $F_\mu^{-1}$ cannot be written as
$F_\nu$ for a probability measure $\nu$ on $\real$.

\item[(iv)] The measure $\mu_{s,r}^\alpha$ satisfies self-similarity with
respect to $s$ as follows.
If $\mu$ is a probability distribution of a random variable $X$, then
let $D_c\mu$ denote the distribution of $cX$.
For $c>0$, we have
\[
\mu_{cs,r}^\alpha= D_{c^{1/\alpha}}\mu_{s,r}^\alpha.
\]
\end{longlist}
\end{rem}

\section{A relation to monotone stable and free Poisson laws}
Let $a_{s}^\alpha$ be a monotone (strictly) $\alpha$-stable
distribution \cite{Has2} characterized by
\[
F_{a_{s}^\alpha}(z)= \bigl(z^\alpha+ (-1)^{\alpha-1}s
\bigr)^{1/\alpha}, \qquad z \in\comp_+,
\]
where $(\alpha,s)$ satisfies the admissible condition. $a_s^2$ is the
centered arcsine law with variance $s/2$ and $a_{s}^1$ is a Cauchy
distribution or a delta measure. The following properties are valuable
to note here.
\begin{longlist}[(1)]
\item[(1)] $a_{s}^\alpha$ is supported on $[0,\infty)$ if and only if $0 <
\alpha\leq1$ and $\arg s =\pi$.

\item[(2)] $a_{s}^\alpha$ is symmetric if and only if $\arg s =(1-\frac{\alpha
}{2})\pi$.

\item[(3)] Both $a_{Ri}^{1/2}$ and $a_{-R}^{1/2}$ are free $\frac
{1}{2}$-stable distributions, but not strictly stable.
\end{longlist}

The proofs are as follows. Let $s:=r\mathrm{e}^{\mathrm{i}\theta}$, $r>0$.
From the Stieltjes inversion formula, the density $p_s^\alpha(x)$ of
$a_s^\alpha$ is given by
\[
p_s^\alpha(x)= %
\cases{\displaystyle \frac{\sin[({1}/{\alpha})\arg(|x|^\alpha+ r\mathrm{e}^{\mathrm{i}(\alpha
\pi- \pi+ \theta)})]}{\pi(|x|^{2\alpha}-2r|x|^\alpha\cos(\alpha\pi
+ \theta) +r^2)^{1/(2\alpha)}},
&\quad $x>0,$ \vspace*{2pt}
\cr
\displaystyle\frac{\sin[({1}/{\alpha})\arg(|x|^\alpha+ r\mathrm{e}^{\mathrm{i}(\pi-\theta)})]}{\pi
(|x|^{2\alpha}-2r|x|^\alpha\cos\theta+r^2)^{1/(2\alpha)}}, &\quad $x<0,$ } %
\]
where $\arg z$ is defined in $(\mathbb{C}_+\cup\mathbb{R})\setminus\{
0\}$ so that it takes values in $[0,\pi]$. Now the properties (1) and
(2) can be proved easily.

(3) It was proved in \cite{Be-Vo} that free $\frac{1}{2}$-stable
distributions are characterized in terms of the Voiculescu transform
$\phi_\mu(z)=bz^{1/2}+c$, where $c\in\mathbb{R}$ and $\arg b\in[\pi
,3\pi/2]$. Moreover, strictly stable laws correspond to the case $c=0$.
Since $F_{a_{s}^{1/2}}(z)= (z^{1/2}-is)^{2}$, we have $\phi_{a_{s}^{1/2}}(z)=F^{-1}_{a_s^{1/2}}(z)-z=2\mathrm{i}sz^{1/2}-s^2$, which for
$s=-R$ or $s=\mathrm{i}R$ means that $a_{s}^{1/2}$ is free $\frac{1}{2}$-stable,
but not strictly stable.



The main theorem of this section is the following.
%
\begin{thmm}
$\mu_{s,2}^\alpha$ is a free compound Poisson distribution for any
admissible $(\alpha,s)$. Moreover, its L\'{e}vy measure $\nu_{s,2}^\alpha$ is given by the monotone stable distribution
$a_{s/4}^\alpha$.
\end{thmm}
\begin{pf}
Thanks to Proposition 4 of \cite{P-S}, it suffices to prove that
$R^\alpha_{s,2}(z) =\psi_{a_{s/4}^\alpha}(z)$, or equivalently, $\phi^\alpha_{s,2}(z) = z^2 G_{a_{s/4}^\alpha}(z)-z$, in an open set of the
form $\Gamma_{\eta,M}$.

As in (\ref{eq7}), a naive relation $(zw)^\alpha= z^\alpha w^\alpha$
may not be valid. To avoid this problem, we understand that $ ( 1
-\frac{s}{4} (-\frac{1}{z} )^\alpha )^{1/\alpha}$,
appearing below, is defined by using the generalized binomial expansion
$(1+w)^{1/\alpha}=\sum_{n=0}^\infty{}_{1/\alpha}C_n w^n$ for $|w|<1$.
Then for any $\eta>0$, the following calculation is correct in $\Gamma_{\eta,M}$ with large $M>0$:
\begin{eqnarray*}
F_{a_{s/4}^\alpha}(z) &=& \biggl(z^\alpha+ \frac{s}{4}(-1)^{\alpha-1}
\biggr)^{1/\alpha}
\\
&=& \biggl( z^\alpha-\frac{s}{4}(-1)^\alpha
\biggr)^{1/\alpha}
\\
&=& \biggl( z^\alpha \biggl( 1 -\frac{s}{4z^\alpha}(-1)^\alpha
\biggr) \biggr)^{1/\alpha}
\\
&=& \biggl( z^\alpha \biggl( 1 -\frac{s}{4} \biggl(-
\frac{1}{z} \biggr)^\alpha \biggr) \biggr)^{1/\alpha}
\\
&=& z \biggl( 1 -\frac{s}{4} \biggl(-\frac{1}{z}
\biggr)^\alpha \biggr)^{1/\alpha}.
\end{eqnarray*}
Therefore,
\[
z^2 G_{a_{s/4}^\alpha}(z)-z = \frac{z}{ ( 1 -({s}/{4}) (-
{1}/{z} )^\alpha )^{1/\alpha}} -z.
\]
On the other hand, the Voiculescu transform of $\mu_{s,2}^\alpha$ is
given as
\begin{eqnarray*}
\phi^\alpha_{s,2}(z) &=& F^{\alpha}_{s/2, 1/2}(z)-z
\\
&=& -\frac{1}{ ( ({1- (1-({s}/{2})(-{1}/{z})^\alpha
)^2})/{s}  )^{1/\alpha}}-z
\\
&=& -\frac{1}{ ( ({s(-{1}/{z})^\alpha- ({s^2}/{4}) (-
{1}/{z})^{2\alpha}})/{s}  )^{1/\alpha}} -z
\\
&=& -\frac{1}{ ((-{1}/{z})^\alpha (1 - ({s}/{4}) (-
{1}/{z})^{\alpha} )  )^{1/\alpha}} -z
\\
&=&\frac{z}{ ( 1 -({s}/{4}) (-{1}/{z} )^\alpha
)^{1/\alpha}} -z
\end{eqnarray*}
in $\Gamma_{\eta,M}$. Therefore, we have proved $\phi^\alpha_{s,2}(z) =
z^2 G_{a_{s/4}^\alpha}(z)-z$.
\end{pf}
With Proposition 4 of \cite{P-S}, the above result implies $\mu_{s,2}^\alpha= m \boxtimes a_{s/4}^\alpha$ if $\mu_{s,2}^\alpha$ and
$a_{s/4}^\alpha$ are symmetric or supported on $[0,\infty)$. We do not
know if this holds for any admissible pair $(\alpha,s)$ since
$S$-transforms are not defined for probability measures which are not
symmetric or supported on $[0,\infty)$.



\begin{thmm}\label{thm22} Let $(\alpha,s)$ satisfy either of the
following conditions: $0 < \alpha\leq1$ and $\arg s \in\{(1-\alpha
/2)\pi, \pi\}$; $1 < \alpha\leq2$ and $\arg s =(1-\alpha/2)\pi$.
Then $\mu_{s,2}^\alpha= m \boxtimes a_{s/4}^\alpha$.
\end{thmm}

\begin{exa}
In general, the density of $\mu^\alpha_{s,2}$ is difficult to
calculate. In some cases, however, the density is explicit as we show below.
\begin{longlist}[(1)]
\item[(1)] Let us consider $(\alpha, s, r) = (1, \mathrm{i}, 2)$. Then $\mu_{\mathrm{i},2}^1$ is
the free multiplicative convolution of the Marchenko--Pastur law and a
symmetric Cauchy distribution. This is absolutely continuous with a
strictly positive density on $\real$ written as
\[
\frac{\sqrt{2}}{\pi} \biggl( \sqrt{1 + \sqrt{1 + \frac{1}{x^2}}} -\sqrt {2}
\biggr).
\]
We mention that this probability measure belongs to a class proposed in
\cite{Has5}.

\item[(2)] Let $(\alpha, s, r)=(\frac{1}{2}, -1, 2)$. Then the corresponding
probability measure is supported on $[0,\infty)$ with a density
\[
\frac{4\sqrt{2}}{\pi} \biggl( \frac{1}{\sqrt{2x}}-\sqrt{-1+\sqrt{1+
\frac
{1}{x}}} \biggr).
\]

\item[(3)] As shown in \cite{ABNPA09}, $\mu_{s,2}^2$ for $s >0$ is a symmetric
beta distribution:
\[
\mu_{s,2}^2(\mathrm{d}x) = \frac{1}{\pi\sqrt{s}}|x|^{-1/2}\bigl(\sqrt
{s}-|x|\bigr)^{1/2}\,\mathrm{d}x, \qquad -\sqrt{s} \leq x \leq\sqrt{s}.
\]
\end{longlist}
\end{exa}

In addition to $\mu_{s,2}^\alpha$, some monotone stable distributions
are also $\boxplus$-infinitely divisible. This property was essentially
proved by Biane \cite{Biane}.

\begin{prop}
$a_s^\alpha$ is $\boxplus$-infinitely divisible if and only if $(\alpha
,s)$ satisfies either of the following conditions:
\begin{longlist}[(1)]
\item[(1)] $\frac{1}{2} \leq\alpha< 1$ and $\arg s \in\{(1-\alpha)\pi,\pi\}$;

\item[(2)] $\alpha=1$.
\end{longlist}
\end{prop}
In fact, Biane considered only special values for $\arg s$, but the
same proof can be applied to the above result.

Finally, we note the $S$-transforms of $\mu^\alpha_{s,2}$ and
$a_s^\alpha$.
%
\begin{prop}\label{prop21} Let $(\alpha,s)$ satisfy either of the
following conditions:
$0 < \alpha\leq1$ and $\arg s \in\{(1-\alpha/2)\pi, \pi\}$; $1 <
\alpha\leq2$ and $\arg s =(1-\alpha/2)\pi$. Then
\begin{longlist}[(1)]
\item[(1)] $S_{a_{s}^\alpha}(z) = -\frac{1}{z} ( \frac{(1+z)^\alpha
-1}{s} )^{1/\alpha}$,   $z \in(-1,0)$,

\item[(2)] $S_{\mu_{s,2}^\alpha}(z)= -\frac{4^{1/\alpha}}{z(z+1)} ( \frac
{(1+z)^\alpha-1}{s} )^{1/\alpha} = S_m(z)S_{a^\alpha
_{s/4}}(z)$,  $z \in(-1,0)$.
\end{longlist}
\end{prop}
\begin{pf}
The Voiculescu transform $\phi_{a_s^\alpha}$ can be calculated as $\phi_{a_s^\alpha}(w) = F_{a_s^\alpha}^{-1}(w)-w = (w^\alpha+ (-1)^\alpha
s)^{1/\alpha} -w$.
Let us define $z:=R_{a_s^\alpha}(w) = w\phi_{a_s^\alpha}(\frac{1}{w})$. Then
$(1+z)^\alpha= 1 +s(-w)^\alpha$. Since $R_{a_s^\alpha}(zS_{a_s^\alpha
}(z))=z$ holds, the desired formula follows. A similar calculation is
possible for $\mu_{s,2}^\alpha$.
\end{pf}

\section{\texorpdfstring{More on free infinite divisibility of $\mu_{s,r}^\alpha$}{More on free infinite divisibility of mu s,r alpha}}\label{sec3}
In the previous section, we proved that $\mu_{s,r}^\alpha$ is $\boxplus
$-infinitely divisible whenever $r =2$. In this section we will
determine infinite divisibility for $r \neq2$. We found the general
case is too difficult to treat, so that we only consider the problem
for some parameters. The main results of this section are the following.
\begin{longlist}[(1)]
\item[(1)] If $0 < \alpha\leq1$ and $1 \leq r \leq2$, then $\mu_{s,r}^\alpha$ is $\boxplus$-infinitely divisible.
\item[(2)] If $1 \leq\alpha\leq2$ and $1 \leq r \leq\frac{2}{\alpha
}$, then $\mu_{s,r}^\alpha$ is $\boxplus$-infinitely divisible.
\item[(3)] $\mu_{s,3}^1$ is $\boxplus$-infinitely divisible if and only
if $\arg s = \frac{\pi}{2}$.
\item[(4)] If $\alpha> 1$, there exists an $r_0=r_0(\alpha, s) > 1$
such that $\mu_{s,r}^\alpha$ is not $\boxplus$-infinitely divisible for
$r > r_0$.
\end{longlist}
We also show that some beta distributions are $\boxplus$-infinitely
divisible, and some are not.

\subsection{\texorpdfstring{The case $1\leq r\leq2$}{The case 1<=r<=2}} \label{subsec41}
To prove the free infinite divisibility of $\mu_{s,r}^\alpha$, we
introduce a subclass of $\boxplus$-infinitely divisible distributions.
%
\begin{defi}
A probability measure $\mu$ is said to be in class $\iu$\footnote{The
symbol $\iu$ stands for \textit{univalent inverse reciprocal Cauchy
transforms}.} if $F_\mu$ is univalent in $\comp_+$ and, moreover, $F_\mu^{-1}$ has an analytic continuation from $F_\mu(\comp_+)$ to $\comp_+$
as a univalent function.
\end{defi}
The following property was implicitly used in \cite{B}.
%
\begin{prop}\label{lem1}
$\mu\in\iu$ implies that $\mu$ is $\boxplus$-infinitely divisible.
\end{prop}
\begin{pf}
The Voiculescu transform $\phi_\mu$ has an analytic continuation to
$\comp_+$ defined by $F_\mu^{-1}(z) -z$. If there existed a point $z_0
\in\comp_+$ such that $\im\phi_\mu(z_0) > 0$, then $\im F_\mu^{-1}(z_0) = \im(z_0 + \phi_\mu(z_0)) > \im z_0 >0$. Since $\im F_\mu
(w) \geq\im w$ for $w \in\comp_+ $, $\im F^{-1}_\mu(z) \leq\im z$
for $z \in F_\mu(\comp_+)$. Therefore, $z_0$ never belongs to $F_\mu
(\comp_+)$. However, since $F^{-1}_\mu$ is univalent in $\comp_+$ and
$F^{-1}_\mu(F_\mu(\comp_+)) = \comp_+$, $z_0$ must satisfy $\im F_\mu^{-1}(z_0) \leq0$, which contradicts the inequality $\im F_\mu^{-1}(z_0) > 0$. Therefore, $\phi_\mu$ maps $\comp_+$ into $\comp_-
\cup\real$.
\end{pf}
%
\begin{rem}If $\mu$ is $\boxplus$-infinitely divisible, then $F_\mu$ is
always univalent in $\comp_+$. This can be proved for instance by using
the so-called subordination functions. Let $\mu$ be $\boxplus
$-infinitely divisible and $\mu_t=\mu^{\boxplus t}$ be the probability
measure corresponding to the Voiculescu transform $t\phi_\mu$.
For $s \leq t$, an analytic function $\omega_{s,t}\dvt  \mathbb{C}_+\to
\mathbb{C}_+$ exists so that it satisfies $F_{\mu_s} \circ\omega_{s,t}
= F_{\mu_t}$. $\omega_{s,t}$ is called a subordination function. The
reader is referred also to equation (5.4) of \cite{Bel3}, where the
following replacements are required: $\mu$ by $\mu^{\boxplus s}$ and
$t$ by $t/s$. The relation $F_{\mu_s} \circ\omega_{s,t} = F_{\mu_t}$
is equivalent to
%
\begin{equation}
\label{eq111} F_{\mu_t}(z)= \frac{t/s}{t/s-1}\omega_{s,t}(z)-
\frac{z}{t/s-1}.
\end{equation}
Moreover, it is proved in Theorem 4.6 of \cite{Bel3} that
\[
\bigl|\omega_{s,t}(z_1)-\omega_{s,t}(z_2)\bigr|
\geq\tfrac{1}{2}|z_1 - z_2|, \qquad z_1,z_2
\in\mathbb{C}_+.
\]
Taking the limit $s \to0$ in (\ref{eq111}), we get
\[
\bigl|F_{\mu_t}(z_1)-F_{\mu_t}(z_2)\bigr| \geq
\tfrac{1}{2}|z_1 - z_2|,\qquad  z_1,z_2
\in\mathbb{C}_+,
\]
so that $F_{\mu_t}$ is univalent in $\comp_+$.

\end{rem}
For instance, the normal law $\frac{1}{\sqrt{2\pi}}\mathrm{e}^{-x^2/2}\,\mathrm{d}x$ is in
$\iu$ from the result of \cite{B}. Moreover, we can easily prove that
Wigner's semicircle law, the Marchenko--Pastur law and the Cauchy
distribution belong to $\iu$.

$\iu$ is closed under the weak topology. This is proved as follows. The
convergence of $\mu_n$ implies the local uniform convergence of the
Voiculescu transforms $\phi_{\mu_n}$ \cite{Be-Vo}. Since $F^{-1}_{\mu
_n}(z) = z + \phi_{\mu_n}(z)$ converges locally uniformly, the limit
function is univalent. Also $F_{\mu_n}$ itself converges to a univalent
function. Therefore, the limit measure belongs to the class $\iu$.

We note that $\iu$ is a proper subset of all $\boxplus$-infinitely
divisible distributions. For instance, let $\mu$ be a probability
measure characterized by the Voiculescu transform $\phi_\mu(z) = \frac
{1}{z-1}+\frac{1}{z+1}$. Then $F^{-1}_\mu(z) = z + \phi_\mu(z) = z+\frac
{1}{z-1}+\frac{1}{z+1}$. We can find two distinct points $z_1,z_2$ such
that $F_\mu^{-1}(z_1) = F_\mu^{-1}(z_2)$ with $z_1 = \mathrm{i}y$ for small $y
> 0$ and $z_2$ near to $\mathrm{i}$. This example also proves that $\iu$ is not
closed under the free convolution, since
the measures $\nu$ and $\lambda$, respectively, defined by $\phi_\nu
(z):= \frac{1}{z-1}$ and $\phi_\lambda(z):=\frac{1}{z+1}$, both belong
to $\iu$.

From Theorem \ref{thm1}, the map $F_{s,r}^\alpha$ is univalent in $\comp_+$ for any admissible $(\alpha, s)$ and $r \geq1$, so that we only
have to prove the inverse
$(F_{s,r}^\alpha)^{-1}$ is univalent in $\comp_+$.

\begin{thmm}\label{thm21}
Let $(\alpha,s)$ be an admissible pair. Then $\mu_{s,r}^\alpha\in\iu$
if either of the following conditions holds:
\begin{longlist}[(1)]
\item[(1)] $0 < \alpha\leq1$ and $1 \leq r \leq2$;

\item[(2)] $1 \leq\alpha\leq2$ and $1 \leq r \leq\frac{2}{\alpha}$.
\end{longlist}
\end{thmm}
\begin{pf}
By Remark \ref{rem2}, the explicit formula for $(F^{\alpha
}_{s,r})^{-1}(z)$ is
\[
\bigl(F^{\alpha}_{s,r}\bigr)^{-1}(z) = -
\frac{1}{ ({(1-(1-
({s}/{r})(-{1}/{z})^\alpha)^r)}/{s} )^{1/\alpha}}.
\]
Let us define $\theta:= \arg s$ and $E_{s,r}^\alpha(z):= \frac
{1-(1-({s}/{r})(-{1}/{z})^\alpha)^r}{s}$. First, we consider $1
\leq\alpha\leq2$. Since the image of the function $1-\frac
{s}{r}(-\frac{1}{z})^\alpha$ for $z \in\comp_+$ is contained in the
sector $\{z \in\comp\dvt  z \neq0,  -(\pi- \theta) < \arg z < -(\pi
-\theta) +\alpha\pi\}$, one can easily see that $E_{s,r}^\alpha(\comp_+)^c$ contains a half line starting from $0$. In particular,
$E_{s,r}^\alpha$ is univalent in $\comp_+$. Therefore, we can take that
line as a slit for the function $z \mapsto z^{1/\alpha}$, which then
becomes univalent outside the slit.

Let us focus on the case $0 < \alpha\leq1$. If $1 \leq r \leq2$, one
can prove that $E_{s,r}^\alpha(\comp_+)$ is contained in a sector with
central angle $r\alpha\pi$ and therefore $E_{s,r}^\alpha$ is univalent
in $\comp_+$. Since $r\alpha\pi\leq2\alpha\pi$, the map $z \mapsto
z^{1/\alpha}$ can be defined as a univalent map in that sector.
\end{pf}

We take $\alpha=1$ and $s = -1$ as a special case. Then $\mu_{-1, r}^1$
is the beta distribution $B(1-\frac{1}{r},1+\frac{1}{r})$ for $1 < r <
\infty$:
\[
\mu_{-1, r}^1(\mathrm{d}x) = \frac{r\sin(\pi/r)}{\pi} x^{-1/r}(1-x)^{1/r}\,\mathrm{d}x,\qquad
0 < x < 1.
\]
Indeed, now we have $G_{-1,r}^1(z) = r (1- (1-\frac{1}{z}
)^{1/r} )$. It holds that
\[
\lim_{y \searrow0} \operatorname{Im} G_{-1,r}^1(x+\mathrm{i}y)=0
\]
if $x >1$ or $x<0$ and
\[
\lim_{y \searrow0} \operatorname{Im} G_{-1,r}^1(x+\mathrm{i}y)=-r
\operatorname{Im} \biggl( \mathrm{e}^{\mathrm{i}\pi/r} \biggl(\frac{1-x}{x}
\biggr)^{1/r} \biggr) = -r \sin(\pi/r) \biggl(\frac{1-x}{x}
\biggr)^{1/r}
\]
if $x \in(0,1)$.
The Stieltjes inversion formula \cite{Akh} $\mu(\mathrm{d}x) = -\frac{1}{\pi}\lim_{y \searrow0} \operatorname{Im} G_{-1,r}^1(x+\mathrm{i}y)\,\mathrm{d}x$ implies the conclusion.

A consequence of Theorem \ref{thm21} is that the beta distribution
$B(1-\frac{1}{r},1+\frac{1}{r})$ is $\boxplus$-infinitely divisible for
$1 < r \leq2$. More strongly, we can prove the following.

\begin{thmm}
The beta distribution $B(1-\frac{1}{r},1+\frac{1}{r})$ $(1 < r < \infty
)$ is $\boxplus$-infinitely divisible if and only if $1 < r \leq2$.
The L\'{e}vy measure $\nu_{-1,r}^1$ for $1 < r < 2$ can be calculated as
\[
\label{eq44} \nu_{-1,r}^1(\mathrm{d}x) = \frac{|\sin(r\pi)|}{\pi}
\frac{x^{r-2}(1/r
-x)^r}{(1/r -x)^{2r}- 2x^r (1/r -x)^r \cos(r\pi) + x^{2r}}\,\mathrm{d}x, \qquad 0 < x < \frac{1}{r}.
\]
\end{thmm}
\begin{pf}
By Remark \ref{rem2}(ii), $(F_{-1,r}^1)^{-1}$ is calculated as
\[
\bigl(F_{-1,r}^1\bigr)^{-1}(z) = \biggl(1 -
\biggl(1- \frac{1}{rz} \biggr)^r \biggr)^{-1}.
\]
If $r >2$, the function $1 -  (1- \frac{1}{rz} )^r$ has a zero
point in the upper half-plane, so that $(F_{-1,r}^1)^{-1}$ never be
defined as an analytic function.
If $r\leq2$, $(F_{-1,r}^1)^{-1}$ is analytic and univalent in the
upper half-plane.

For the L\'{e}vy measure, the Voiculescu transform is $\phi_{-1,r}^1(z)=  (1 -  (1- \frac{1}{rz} )^r )^{-1} -z$.
It holds that $\operatorname{Im} ( 1-  (1-\frac{1}{r(x+\mathrm{i}y)} )^r
) \to0$ as $y \searrow0$ if $x > 1/r$ or $x<0$ and that
\[
1- \biggl(1-\frac{1}{r(x+\mathrm{i}y)} \biggr)^r \to1- \mathrm{e}^{\mathrm{i} r \pi}
\biggl(\frac
{x-1/r}{x} \biggr)^r,\qquad  x\in(0,1/r)
\]
as $y \searrow0$. After some more calculations, one can see
\begin{eqnarray*}
\tau(\mathrm{d}x) &=& -\frac{1}{\pi(1+x^2)}\lim_{y \searrow0}\operatorname{Im}
\phi_{-1,r}^1(x+\mathrm{i}y)\,\mathrm{d}x \\
&=& %
\cases{\displaystyle \frac{1}{\pi(1+x^2)}
\frac{|\sin(r\pi)|x^{r}(1/r
-x)^r}{(1/r -x)^{2r}- 2x^r (1/r -x)^r \cos(r\pi) + x^{2r}}\,\mathrm{d}x,& \quad  $x\in (0, 1/r),$ \vspace*{2pt}
\cr
\displaystyle 0, &\quad $\mbox{otherwise},$ }
\end{eqnarray*}
where $\tau$ is the measure in (\ref{eqtau}). $\tau$ does not have an
atom since
$\lim_{y \searrow0} \mathrm{i}y \phi_\mu(x+\mathrm{i}y)=0$ for any $x \in\real$. The L\'
{e}vy measure $\nu_{-1,r}^1$ is equal to $\frac{1+x^2}{x^2}\tau$ as
explained in Section \ref{sec22}.
\end{pf}

If $s=\operatorname{Re}^{\mathrm{i}\theta}$ is not real, the support of $\mu_{s,r}^1$ is
unbounded. The density for large $|x|$ can be calculated as
\[
\mu_{\operatorname{Re}^{\mathrm{i}\theta},r}^1|_{|x| >R}(\mathrm{d}x) = -\frac{r}{\pi}\sum
_{n=1}^\infty \pmatrix{1/r
\cr
n+1}
\frac{R^n\sin(n\theta)}{x^{n+1}}\,\mathrm{d}x.
\]
In particular, $\mu_{s,r}^1$ belongs to a class introduced in \cite{Has5}.

\subsection{\texorpdfstring{The case $\alpha=1,  r=3$}{The case alpha=1, r=3}}
In Section \ref{subsec41}, the free infinite divisibility of $\mu^\alpha_{s,r}$
was proved for some parameters in terms of the class $\iu$.
In Section \ref{sec2}, we succeeded in proving the free infinite
divisibility of $\mu^\alpha_{s,2}$ since the Voiculescu transform had a
quite explicit form. For other parameters, it is difficult to
investigate the free infinite divisibility. A possible case is for
$\alpha=1$ and $r = 3$. In this case, the Voiculescu transform has a
quite explicit form as in the case $r=2$ and $\boxplus$-infinite
divisibility can be determined completely. Indeed, the Voiculescu
transform is
\[
\phi_{3s,3}^1(z) = \frac{-3s}{1-(1+s/z)^3} -z =
\frac{-3sz^2
-s^2z}{3z^2 +3zs +s^2}.
\]
In contrast to the case $r=2$, infinite divisibility depends on the
parameter $s$ if $r = 3$.
%
\begin{thmm} Let $ 0 \leq\arg s \leq\pi$. Then
$\mu_{s,3}^1$ is $\boxplus$-infinitely divisible if and only if $\arg s
= \frac{\pi}{2}$. The L\'{e}vy measure $\nu_{3\mathrm{i},3}^1$ for $\mu_{3\mathrm{i},
3}^{1}$ can be calculated as
\[
\nu_{3\mathrm{i},3}^1(\mathrm{d}x) = \frac{9x^2}{\pi(9x^4 + 3x^2 +1)}\,\mathrm{d}x, \qquad x \in\real.
\]
\end{thmm}
\begin{pf}
Because of Remark \ref{rem2}(iv), let us consider $s=\mathrm{e}^{\mathrm{i}\theta}$ for
simplicity. After some calculations, we get
\[
\im\phi_{3s,3}^1(x + \mathrm{i}0) = -\frac{9x^4 \sin\theta+3x^3 \sin2\theta
}{|3x^2 + 3xs + s^2|^2}.
\]
Therefore, if $\theta\neq0,\pi,\frac{\pi}{2}$, we can find a point
$x_0 \in\real$ such that $\im\phi_{3s,3}(x_0 +\mathrm{i}0) >0$.
If $\theta= \pi$, we can calculate
\[
\im\phi_{-3,3}^1(x + \mathrm{i}y) = -\frac{y[6y^2+6(x-1/2)^2 -1/2]}{|3x^2 + 3xs
+ s^2|^2},
\]
and therefore $\phi_{3s,3}$ takes a positive value at a point. By
symmetry, also $\phi_{3,3}$ can take a positive value. Therefore, $\mu_{3s,3}$ is not $\boxplus$-infinitely divisible for $\theta\neq\frac
{\pi}{2}$.

For $\theta=\frac{\pi}{2}$, after some calculations, it holds that
\[
\im\phi_{3\mathrm{i},3}^1(x+\mathrm{i}y) = -\frac{9x^4
+18x^2y^2+9y^4+12x^2y+12y^3+6y^2+y}{|3x^2 + 3xi -1|^2} <0,
\]
so that $\mu_{3\mathrm{i},3}^1$ is $\boxplus$-infinitely divisible. The measure
$\tau$ in (\ref{eqtau}) is absolutely continuous with respect to the
Lebesgue measure since
\[
-\frac{1}{\pi}\operatorname{Im} \phi_{3\mathrm{i},3}^1(x+\mathrm{i}y)
\to\frac{9x^4}{\pi
|3x^2+3ix-1|^2} = \frac{9x^4}{\pi(9x^4 + 3x^2+1)}
\]
locally uniformly in $\real$ as $y \searrow0$. The L\'{e}vy measure is
given by $\frac{1+x^2}{x^2}\tau$, where $\tau$ is defined in~(\ref{eqtau}).
\end{pf}

\subsection{\texorpdfstring{Noninfinite divisibility for $1<\alpha\leq2$ and large $r$}
{Noninfinite divisibility for 1<alpha<=2 and large r}}
We prove the following.
%
\begin{prop} For $\alpha> 1$ and $\arg s \in[0,(2-\alpha)\pi]$, there
exists an $r_0=r_0(\alpha, s) > 1$ such that $\mu_{s,r}^\alpha$ is not
$\boxplus$-infinitely divisible for $r > r_0$.
\end{prop}
\begin{pf}
Let $\theta:=\arg s$. It is sufficient to find a zero point of the
function $E_{s,r}^\alpha(z):=  \frac{1-(1-({s}/{r})(-
{1}/{z})^\alpha)^r}{s}$. The function
$1-\frac{s}{r}(-\frac{1}{z})^\alpha$ maps $\comp_+$ to a shifted sector
$\Omega:=\{z \in\comp\dvt  z \neq0,  -(\pi- \theta) < \arg(z-1) < -(\pi
-\theta) +\alpha\pi\}$. If $\alpha>1$, $\Omega$ and the unit circle
$\{z \in\comp\dvt |z| =1\}$ have intersection which is an arc with an end
point $1$. Let us denote by $\varphi\in(-\pi,\pi)\setminus\{0\}$ the
angle of the other end point of that arc. We can take $r_0(\alpha,s)$
to be~$\frac{2\pi}{|\varphi|}$.
\end{pf}

\section*{Acknowledgements}
This work was initiated during the authors' visit to CIMAT, thanks to
the hospitality of Professor V\'{i}ctor P\'{e}rez-Abreu, who also
suggested many improvements of the paper. Octavio Arizmendi was
supported by funds of R. Speicher from the Alfried Krupp von Bohlen und
Halbach Stiftung. Takahiro Hasebe was supported by Global COE Program
at Kyoto University.
%

%


\printhistory


\begin{thebibliography}{21}

\bibitem{Akh}
\begin{bbook}[auto:STB|2012/09/25|13:49:33]
\bauthor{\bsnm{Akhiezer},~\bfnm{N.~I.}\binits{N.I.}}
(\byear{1965}).
\btitle{The Classical Moment Problem}.
\baddress{Edinburgh}: \bpublisher{Oliver \& Boyd}.
\bptok{imsref}%
\end{bbook}
\endbibitem

\bibitem{A-P}
\begin{barticle}[mr]
\bauthor{\bsnm{Arizmendi},~\bfnm{Octavio}\binits{O.}} \AND
\bauthor{\bsnm{P{\'e}rez-Abreu},~\bfnm{Victor}\binits{V.}}
(\byear{2009}).
\btitle{The {$S$}-transform of symmetric probability measures with unbounded
supports}.
\bjournal{Proc. Amer. Math. Soc.}
\bvolume{137}
\bpages{3057--3066}.
\bid{doi={10.1090/S0002-9939-09-09841-4}, issn={0002-9939}, mr={2506464}}
\bptok{imsref}%
\end{barticle}
\endbibitem

\bibitem{ABNPA09}
\begin{barticle}[mr]
\bauthor{\bsnm{Arizmendi},~\bfnm{Octavio}\binits{O.}},
\bauthor{\bsnm{Barndorff-Nielsen},~\bfnm{Ole~E.}\binits{O.E.}} \AND
\bauthor{\bsnm{P{\'e}rez-Abreu},~\bfnm{V{\'{\i}}ctor}\binits{V.}}
(\byear{2010}).
\btitle{On free and classical type {$G$} distributions}.
\bjournal{Braz. J. Probab. Stat.}
\bvolume{24}
\bpages{106--127}.
\bid{doi={10.1214/09-BJPS039}, issn={0103-0752}, mr={2643561}}
\bptok{imsref}%
\end{barticle}
\endbibitem


\bibitem{BNT06}
\begin{bincollection}[mr]
\bauthor{\bsnm{Barndorff-Nielsen},~\bfnm{Ole~E.}\binits{O.E.}} \AND
\bauthor{\bsnm{Thorbj{\o}rnsen},~\bfnm{Steen}\binits{S.}}
(\byear{2006}).
\btitle{Classical and free infinite divisibility and {L}\'evy processes}.
In \bbooktitle{Quantum Independent Increment Processes {II}}
(\beditor{\bfnm{U.}\binits{U.}~\bsnm{Franz}}
\AND
\beditor{\bfnm{M.}\binits{M.}~\bsnm{Sch\"{u}rmann}}, eds.).
\bseries{Lecture Notes in Math.}
\bvolume{1866}
\bpages{33--159}.
\baddress{Berlin}: \bpublisher{Springer}.
\bid{doi={10.1007/11376637_2}, mr={2213448}}
\bptok{imsref}%
\end{bincollection}
\endbibitem

\bibitem{Bel3}
\begin{barticle}[mr]
\bauthor{\bsnm{Belinschi},~\bfnm{S.~T.}\binits{S.T.}} \AND
\bauthor{\bsnm{Bercovici},~\bfnm{H.}\binits{H.}}
(\byear{2005}).
\btitle{Partially defined semigroups relative to multiplicative free
convolution}.
\bjournal{Int. Math. Res. Not.}
\bvolume{2}
\bpages{65--101}.
\bid{doi={10.1155/IMRN.2005.65}, issn={1073-7928}, mr={2128863}}
\bptok{imsref}%
\end{barticle}
\endbibitem

\bibitem{B}
\begin{barticle}[mr]
\bauthor{\bsnm{Belinschi},~\bfnm{Serban~T.}\binits{S.T.}},
\bauthor{\bsnm{Bo{\.z}ejko},~\bfnm{Marek}\binits{M.}},
\bauthor{\bsnm{Lehner},~\bfnm{Franz}\binits{F.}} \AND
\bauthor{\bsnm{Speicher},~\bfnm{Roland}\binits{R.}}
(\byear{2011}).
\btitle{The normal distribution is {$\boxplus$}-infinitely divisible}.
\bjournal{Adv. Math.}
\bvolume{226}
\bpages{3677--3698}.
\bid{doi={10.1016/j.aim.2010.10.025}, issn={0001-8708}, mr={2764902}}
\bptok{imsref}%
\end{barticle}
\endbibitem


\bibitem{Be-Vo}
\begin{barticle}[mr]
\bauthor{\bsnm{Bercovici},~\bfnm{Hari}\binits{H.}} \AND
\bauthor{\bsnm{Voiculescu},~\bfnm{Dan}\binits{D.}}
(\byear{1993}).
\btitle{Free convolution of measures with unbounded support}.
\bjournal{Indiana Univ. Math. J.}
\bvolume{42}
\bpages{733--773}.
\bid{doi={10.1512/iumj.1993.42.42033}, issn={0022-2518}, mr={1254116}}
\bptok{imsref}%
\end{barticle}
\endbibitem

\bibitem{Biane}
\begin{barticle}[mr]
\bauthor{\bsnm{Biane},~\bfnm{Philippe}\binits{P.}}
(\byear{1998}).
\btitle{Processes with free increments}.
\bjournal{Math. Z.}
\bvolume{227}
\bpages{143--174}.
\bid{doi={10.1007/PL00004363}, issn={0025-5874}, mr={1605393}}
\bptok{imsref}%
\end{barticle}
\endbibitem

\bibitem{Has2}
\begin{barticle}[mr]
\bauthor{\bsnm{Hasebe},~\bfnm{Takahiro}\binits{T.}}
(\byear{2010}).
\btitle{Monotone convolution and monotone infinite divisibility from complex
analytic viewpoint}.
\bjournal{Infin. Dimens. Anal. Quantum Probab. Relat. Top.}
\bvolume{13}
\bpages{111--131}.
\bid{doi={10.1142/S0219025710003973}, issn={0219-0257}, mr={2646794}}
\bptok{imsref}%
\end{barticle}
\endbibitem

\bibitem{Has5}
\begin{barticle}[mr]
\bauthor{\bsnm{Hasebe},~\bfnm{Takahiro}\binits{T.}}
(\byear{2012}).
\btitle{Analytic {c}ontinuations of {F}ourier and {S}tieltjes {t}ransforms and
{g}eneralized {m}oments of {p}robability {m}easures}.
\bjournal{J. Theoret. Probab.}
\bvolume{25}
\bpages{756--770}.
\bid{doi={10.1007/s10959-011-0344-9}, issn={0894-9840}, mr={2956211}}
\bptok{imsref}%
\end{barticle}
\endbibitem

\bibitem{MP}
\begin{barticle}[mr]
\bauthor{\bsnm{Mar{\v{c}}enko},~\bfnm{V.~A.}\binits{V.A.}} \AND
\bauthor{\bsnm{Pastur},~\bfnm{L.~A.}\binits{L.A.}}
(\byear{1967}).
\btitle{Distribution of eigenvalues in certain sets of random matrices}.
\bjournal{Mat. Sb. (N.S.)}
\bvolume{72 (114)}
\bpages{507--536}.
\bid{mr={0208649}}
\bptok{imsref}%
\end{barticle}
\endbibitem

\bibitem{Mlot}
\begin{barticle}[mr]
\bauthor{\bsnm{M{\l}otkowski},~\bfnm{Wojciech}\binits{W.}}
(\byear{2010}).
\btitle{Fuss--{C}atalan numbers in noncommutative probability}.
\bjournal{Doc. Math.}
\bvolume{15}
\bpages{939--955}.
\bid{issn={1431-0635}, mr={2745687}}
\bptok{imsref}%
\end{barticle}
\endbibitem

\bibitem{Mur2}
\begin{barticle}[mr]
\bauthor{\bsnm{Muraki},~\bfnm{Naofumi}\binits{N.}}
(\byear{2001}).
\btitle{Monotonic independence, monotonic central limit theorem and monotonic
law of small numbers}.
\bjournal{Infin. Dimens. Anal. Quantum Probab. Relat. Top.}
\bvolume{4}
\bpages{39--58}.
\bid{doi={10.1142/S0219025701000339}, issn={0219-0257}, mr={1824472}}
\bptok{imsref}%
\end{barticle}
\endbibitem

\bibitem{P-S}
\begin{barticle}[mr]
\bauthor{\bsnm{P{\'e}rez-Abreu},~\bfnm{Victor}\binits{V.}} \AND
\bauthor{\bsnm{Sakuma},~\bfnm{Noriyoshi}\binits{N.}}
(\byear{2012}).
\btitle{Free infinite divisibility of free multiplicative mixtures of the
{W}igner distribution}.
\bjournal{J. Theoret. Probab.}
\bvolume{25}
\bpages{100--121}.
\bid{doi={10.1007/s10959-010-0288-5}, issn={0894-9840}, mr={2886381}}
\bptok{imsref}%
\end{barticle}
\endbibitem

\bibitem{RS}
\begin{barticle}[mr]
\bauthor{\bsnm{Rao},~\bfnm{N.~Raj}\binits{N.R.}} \AND
\bauthor{\bsnm{Speicher},~\bfnm{Roland}\binits{R.}}
(\byear{2007}).
\btitle{Multiplication of free random variables and the {$S$}-transform: The
case of vanishing mean}.
\bjournal{Electron. Commun. Probab.}
\bvolume{12}
\bpages{248--258}.
\bid{doi={10.1214/ECP.v12-1274}, issn={1083-589X}, mr={2335895}}
\bptok{imsref}%
\end{barticle}
\endbibitem


\bibitem{T}
\begin{bbook}[mr]
\bauthor{\bsnm{Teschl},~\bfnm{Gerald}\binits{G.}}
(\byear{2000}).
\btitle{Jacobi Operators and Completely Integrable Nonlinear Lattices}.
\bseries{Mathematical Surveys and Monographs}
\bvolume{72}.
\baddress{Providence, RI}: \bpublisher{Amer. Math. Soc.}
\bid{mr={1711536}}
\bptok{imsref}%
\end{bbook}
\endbibitem

\bibitem{V2}
\begin{barticle}[mr]
\bauthor{\bsnm{Voiculescu},~\bfnm{Dan}\binits{D.}}
(\byear{1986}).
\btitle{Addition of certain noncommuting random variables}.
\bjournal{J. Funct. Anal.}
\bvolume{66}
\bpages{323--346}.
\bid{doi={10.1016/0022-1236(86)90062-5}, issn={0022-1236}, mr={0839105}}
\bptok{imsref}%
\end{barticle}
\endbibitem

\bibitem{Voi87}
\begin{barticle}[mr]
\bauthor{\bsnm{Voiculescu},~\bfnm{Dan}\binits{D.}}
(\byear{1987}).
\btitle{Multiplication of certain noncommuting random variables}.
\bjournal{J. Operator Theory}
\bvolume{18}
\bpages{223--235}.
\bid{issn={0379-4024}, mr={0915507}}
\bptok{imsref}%
\end{barticle}
\endbibitem

\end{thebibliography}
\end{document}